\documentclass[reqno,A4paper]{amsart}

\usepackage{amsmath}
\usepackage{amssymb}
\usepackage{amsthm}
\usepackage{enumerate}
\usepackage{mathrsfs} %%\mathcal classic,  \mathscr decorative
\usepackage{eqlist}
\usepackage{array}

\theoremstyle{plain}
\newtheorem{theorem}{Theorem}[section]
\newtheorem{lemma}{Lemma}[section]

\theoremstyle{definition}
\newtheorem{definition}{Definition}
\newtheorem{remark}{\textup{Remark}} %\textup for ``Remark'' is required
\newtheorem{example}{\textit{Example}} %\textit for ``Example'' is required

\newtheorem{corollary}{Corollary}
\newtheorem{proposition}{Proposition}%[name of an existing counter to which
\numberwithin{equation}{section}
\newcommand{\cls}{\mathrm{cls}}    %%differential
\begin{document}

\title[Hawaiian homology]%
{On Hawaiian homology groups}
\author[Hamid Torabi, Hanieh Mirebrahimi \and Ameneh Babaee]%
{Hamid Torabi*, Hanieh Mirebrahimi* \and Ameneh Babaee*}

\newcommand{\acr}{\newline\indent}

\address{\llap{*\,}Department of Pure Mathematics, Center of Excellence in Analysis on Algebraic Structures, Ferdowsi University of
Mashhad,\\
P.O.Box 1159-91775, Mashhad, Iran.}
\email{h.torabi@ferdowsi.um.ac.ir}
\email{h\_mirebrahimi@um.ac.ir (H. Mirebrahimi)} \email{am.babaee@mail.um.ac.ir (A. Babaee)}
%\address{\llap{**\,}Department of Mathematics\acr
   %                 Comenius University\acr
      %              Mlynsk\'a Dolina\acr
  %                  SK--842 15 Bratislava\acr
  %                  SLOVAKIA}
%\email{Flora.Schon@fmph.uniba.sk}

%%\acr is not required (if you do not need to see a column);
%%in our style \\ makes a column automatically

%\thanks{This research was supported by a grant from Ferdowsi University of Mashhad--Graduate Studies (No. ).}

 %Secondary is optional

\begin{abstract}
In this paper, we introduce a kind of  homology which we call Hawaiian homology to study and classify pointed topological spaces. The Hawaiian homology group has advantages of Hawaiian groups. Moreover, the first Hawaiian homology group is isomorphic to the abelianization of the first Hawaiian group for path-connected and locally path-connected topological spaces. Since Hawaiian homology has concrete elements and abelian structure, its calculations are more routine. Thus we use  Hawaiian homology groups to compare Hawaiian groups, and then we obtain some information about Hawaiian groups of some wild topological spaces.
\end{abstract}

\maketitle

\section{Introduction and Motivation}

Homology is a well-known useful tool of algebraic topology. Since homology has  concrete interpretation, different fields of science apply it as an algebraic modelling of geometric properties to study their observations mathematically \cite{Ghr}.
Homology is a general way of associating a sequence of algebraic objects, such as abelian groups or modules, with other mathematical objects such as topological spaces. Homology groups were originally defined in algebraic topology and then it is generalized  in a wide variety of other contexts, such as abstract algebra, groups, Lie algebras, Galois theory, and algebraic geometry.

The original motivation for defining homology groups was  to distinguish shapes  by examining their holes.
%For instance, a circle is not a disk because the circle has a hole through it while the disk is solid, and the ordinary sphere is not a circle because the sphere encloses a two-dimensional hole while the circle encloses a one-dimensional hole. However, because a hole is "not there", it is not immediately obvious how to define a hole or how to distinguish different kinds of holes. Homology was originally a rigorous mathematical method for defining and categorizing holes in a manifold.
Homology group has a famous characteristic, the Betti number, which is known as the various counting the number of holes of the spaces, and by this fact, the Betti number helps to classify spaces. The $n$th Betti number is defined as the rank of the $n$th homology group of the given topological space.

There are many different homology theories, each of which has its advantages, applications, and defects \cite{Ada, Rot}. A particular type of mathematical object, such as a topological space or a group, may have one or more associated homology theories. When the underlying object has a geometric interpretation as topological spaces do, the $n$th homology group represents the behavior in the dimension $n$. Most homology groups or modules may be formulated as derived functors on appropriate abelian categories, measuring the failure of a functor to be exact. %From this abstract perspective, homology groups are determined by objects of a derived category.

In Section 2, we define a Hawaiian homology similar to the singular homology theory, with the difference whose simplexes are shrinking sequences of singular simplexes tending to the base point. Thus, it is natural that Hawaiian homology groups depend on the choice of the base point. More precisely, the Hawaiian homology is a covariant functor from the category of pointed topological spaces to the category of abelian groups. We see that Hawaiian homology groups depend on the local behavior of the spaces at their base points. This is a great difference between the Hawaiian homology and the singular homology whose structure depends only on the homotopy kind of  space, while the base point has an essential role in calculating the Hawaiian homology. In fact, the Hawaiian homology is an invariant of those homotopies preserving the base point called pointed homotopies. We present some examples to clarify the effects of choice of the base point and pointed homotopies.

Also, we investigate basic properties of the Hawaiian homology, including the Hawaiian homology of the singletons, disjoint union, homotopic maps, and so on.
Then we present a close relation between the Hawaiian homology group and Hawaiian group of topological spaces. In fact, the first Hawaiian homology group equals the abelianization of the first Hawaiian group for any path-connected and locally path-connected topological space. The  $n$th Hawaiian group  of a pointed topological space was defined by Karimov and  Repov\v{s} \cite{KarRep}  as the set of all pointed  homotopy classes of continuous maps with a group operation coming from the operation
of  $n$th homotopy groups. The $n$th Hawaiian group, defined below, is a covariant functor from the category of pointed topological spaces to the category of groups, denoted by ${\mathcal{H}}_n(X,x_0)$, where $(X, x_0)$ is a pointed topological space.

\begin{definition}[\cite{KarRep}]
Let $(X, x_0)$ be a pointed space, let $[-]$ denote the class of pointed homotopy,  and let $n = 1, 2, \ldots$. Then the $n$th Hawaiian group of $(X, x_0)$ is defined by $\mathcal{H}_n(X, x_0) = \{[f]:\  f:(\mathbb{HE}^n , \theta) \to (X, x_0)\}$ with the multiplication induced by $(f*g)|_{\mathbb{S}_k^n} = f|_{\mathbb{S}_k^n} \cdot g|_{\mathbb{S}_k^n}$ ($k \in \mathbb{N}$) for any $[f], [g] \in \mathcal{H}_n(X, x_0)$, where $\cdot$ denotes the concatenation of $n$-loops.
\end{definition}

The Hawaiian functor has some advantages over  other well-known covariant functors from the category of topological spaces to the category of groups. Unlike homotopy and homology groups, Hawaiian groups do not preserve free homotopy equivalences. As an example, consider the cone space $C(\mathbb{HE}^1)$, where $\mathbb{HE}^1$ denotes the one-dimensional Hawaiian earring. Since cone spaces are contractible, their homotopy, homology, and cohomology groups are trivial \cite{KarRep}, but  the first Hawaiian group of $C(\mathbb{HE}^1)$ is uncountable \cite{KarRep}.
Also, by using this functor, we can study some local behaviors of spaces. For instance, if $X$ has a countable local basis at $x_0$ and  the $n$th Hawaiian group ${\mathcal{H}}_n (X,x_0)$ is countable, then $X$ is locally $n$-simply connected at $x_0$
(see \cite[Theorem 2]{KarRep}).
Moreover, let $X$ have a countable local basis at $x_0$. Then $\mathcal{H}_n(CX, x)$ is trivial if and only if $X$ is locally $n$-simply connected at $x_0$ and it is uncountable otherwise \cite[Corollaries 2.16 and 2.17]{BabMas1}. Hence, unlike homotopy groups, Hawaiian groups  depend on the behavior of the space at the base point. In this regard, there are path-connected spaces  whose Hawaiian groups are not isomorphic at different points, such as the $n$-dimensional Hawaiian earring, where $n \ge 2$ (see \cite[Corollary 2.11]{BabMas1}).
In this paper, we show that the Hawaiian homology groups have the advantages of Hawaiian groups, from the viewpoint of homology.
%For instance, homotopy equivalent spaces do not necessarily have isomorphic Hawaiian homology groups, because Hawaiian homology groups are invariant of homotopies preserving base points. In this regards there are some contractible spaces whose Hawaiian homology groups are not trivial.
Moreover, Hawaiian homology groups have more concrete elements, and also they can be computed by techniques of abelian groups (see theorem \ref{th4.2nn}), while there rarely exist similar techniques to calculate Hawaiian groups. By using these facts, we use Hawaiian homology groups to study and classify pointed topological spaces; for instance, the first Hawaiian homology group of the Harmonic archipelago at the origin is not isomorphic to  the first Hawaiian homology group at any point else. Also since the first Hawaiian homology group for any locally path-connected space is the abelianization of the first Hawaiian group, it can help us to investigate the structure of Hawaiian groups. Therefore, by this fact that the first Hawaiian homology groups are not isomorphic  at different points,  the first Hawaiian groups of the Harmonic archipelago are not isomorphic.

In Section 3, by employing Hawaiian homology groups, we compare Hawaiian groups of some pointed spaces, for instance, the Harmonic archipelago, Griffiths space, Hawaiian earrings, and so on. Moreover, we calculate Hawaiian homology groups of the Harmonic archipelago and prove that the first Hawaiian homology and the first singular homology groups of the Harmonic archipelago are isomorphic. The Harmonic archipelago is an element of a class of spaces called archipelago spaces constructed by attaching loops to some weak join of spaces. Hawaiian earrings belong to the class of weak join spaces defined as follows.

\begin{definition}[\cite{EdaKaw, BarMil}]
The weak join of a family of spaces $\{(X_i, x_i); i \in I\}$, denoted by $(X, x_*) = \widetilde{\bigvee}_{i \in I} (X_i, x_i)$, is the underlying space of wedge space $(\hat{X}, \hat{x}_*) = \bigvee_{i \in I} (X_i, x_i)$ with the weak topology with respect to $X_i$'s, except at the common point. Every open neighborhood in $X$ at $x_*$ is of the form $\bigcup_{i \in F} U_i \cup \bigcup_{i \in I \setminus F} X_i$, where $U_i$ is an open neighborhood at $x_i$ in $X_i \subset X$ and $F$ is a finite subset of $I$ (see \cite[Section 2]{BarMil} and \cite[p. 18]{EdaKaw}).
\end{definition}

In \cite[Definition 2.1]{MasMir}, small $n$-Hawaiian loop was defined as a small map from $\mathbb{HE}^n$ to any space $X$. All spaces are assumed to be first countable. In this paper, by ${\widehat{\qquad}}^p$ we mean the $p$-adic completion of a group and $P$ denotes the set of all primes.

\section{Hawaiian Homology groups}

In this section, we define the Hawaiian homology as a new invariant from the category of pointed topological spaces to the category of abelian groups.
%Since this invariant does not satisfy the homology axioms, it is called the strong Hawaiian homology.
First, we define the Hawaiian simplex in the natural way. That is, the disjoint union of a countably infinite family  of Euclidean  simplexes whose diameters tend to zero together with the origin as the cluster point.

\begin{definition}
Let $\Delta_k^n = \Delta^n/k = \{a \in \mathbb{R}^{n+1}|\ ka \in \Delta^n\}$, for $k=1,2,\ldots$, where $\Delta^n$ denotes the standard Euclidean $n$-simplex, and let $\Delta_0^n = \{ \theta\}$, where $\theta$ denotes the origin. The $n$-Hawaiian simplex  is defined as follows:
\[H\Delta^n = \bigcup_{k = 0}^{\infty} \Delta^n_k.
\]
Moreover, an $n$-Hawaiian simplex in the given space $(X, x_0)$ is a continuous map $\sigma^n : (H\Delta^n, \theta) \to (X, x_0)$. The restriction of $\sigma^n$ on the $k$th simplex $\Delta_k^n$ can be assumed as an $n$-simplex in $X$, denoted by $\sigma_k^n = \sigma^n|_{\Delta_k^n}$.
\end{definition}

Similar to the definition of singular homology,  the group $S_n(X, x_0)$ is defined as the free abelian group generated by all $n$-Hawaiian simplexes in $(X, x_0)$.  To define the boundary homomorphism $\partial_n: S_n(X, x_0) \to S_{n-1} (X, x_0)$, we need to introduce face maps. Let $k \in \mathbb{N}$. Consider $\varepsilon_{i,k}^n : \Delta_k^{n-1} \to \Delta_k^n$ as the map defined by the $i$th face-map on Euclidean $(n-1)$-simplex together with the appropriate coefficient associated with $k \in \mathbb{N}$.
%For $k = 0$, put $\varepsilon_{i,0}^n$ the constant map.
Then let
$\varepsilon_i^n = \bigcup_{k = 0}^{\infty} \varepsilon_{i, k}^{n} : \bigcup_{k =0}^{\infty} \Delta^{n-1}_k \to \bigcup_{k =0}^{\infty} \Delta^{n}_k$ be the map obtained by the union of all $\varepsilon^n_{i,k}$'s for $k \in \mathbb{N}$ and the constant map for $k = 0$.
The mapping $\varepsilon_i^n$ is continuous, because $\Delta_k^{n-1}$ is mapped into $\Delta_k^n$ by a continuous mapping for each $k \in \mathbb{N}$ and then it maps each convergent sequence to a convergent sequence. Then $\varepsilon_i^n$ is called the $i$th face-map of the $n$-Hawaiian simplex. Now the boundary map $\partial_n: S_n(X, x_0) \to S_{n-1} (x, x_0)$  is defined by $\partial_n (\sigma^n) = \sum_{i =0}^n (-1)^i \sigma \circ \varepsilon_i^n$.
Then we have the following chain complex:
\[
\cdots \longrightarrow S_n(X, x_0) \mathop{\longrightarrow}^{\partial_n}
S_{n-1} (X, x_0)
\mathop{\longrightarrow}^{\partial_{n-1}}
\cdots
\longrightarrow
S_1(X, x_0)
\mathop{\longrightarrow}^{\partial_1}
S_0(X, x_0)
\mathop{\longrightarrow}^{\partial_0}
0.
\]

\begin{remark}
To define Hawaiian homology group, we must verify the equality $\partial_n \partial_{n +1} = 0$. It can be checked by the definition of boundary homomorphism as follows: \begin{align*}
\partial_n \partial_{n +1} (\sigma^{n +1}) & = \partial_n \big(\sum_{j =0}^{n+1} (-1)^j \sigma^{n+1} \circ \varepsilon_j^{n+1} \big) \\
& =
\sum_{i, j} (-1)^{i +j} \sigma^{n+1}  \circ \varepsilon^{n+1}_j \circ \varepsilon_i^n\\
& =
 \sum_{i, j} (-1)^{i +j} \bigcup_{k =0}^{\infty} \big( \sigma^{n+1}_k \circ \varepsilon^{n+1}_{j, k} \circ \varepsilon_{i, k}^n \big)\\
 & =
 \bigcup_{k =0}^{\infty}
 \sum_{i, j} (-1)^{i +j}
 \big( \sigma^{n+1}_k \circ \varepsilon^{n+1}_{j, k} \circ \varepsilon_{i, k}^n \big)\\
 & =
 \bigcup_{k =1}^{\infty}
 \partial_n^s \partial_{n+1}^s (\sigma^{n+1}_k) \cup
 \sum_{i, j} (-1)^{i +j}
 \big( \sigma^{n+1}_0 \circ \varepsilon^{n+1}_{j, 0} \circ \varepsilon_{i, 0}^n \big),
\end{align*}
where $\partial^s$ is defined similar to the boundary homomorphism of singular complex; see \cite[Theorem 4.6]{Rot}. Therefore $\partial_n \partial_{n +1} (\sigma^{n +1})$ equals zero because it is the union of zero homomorphisms in singular homology for $k = 1, 2, \ldots$ \cite[Theorem 4.6]{Rot}.
For $k =0$, since $\sigma^{n+1}_0 \circ \varepsilon^{n+1}_{j, 0} \circ \varepsilon_{i, 0}^n$ equals some constant value, namely $b$, we have the following equalities
\begin{align*}
\sum_{i, j} (-1)^{i +j}
 \big( \sigma^{n+1}_0 \circ \varepsilon^{n+1}_{j, 0} \circ \varepsilon_{i, 0}^n \big) &= (-1)^0 b +  (-1)^1b + (-1)^2 b+ \ldots + (-1)^{(n+1) +n} b\\
 & = b \big(1 + (-1) +1 +(-1) + \ldots + 1 + (-1)\big)\\
 & = b (0) = 0.
\end{align*}
 \end{remark}

Now, we are able to define the Hawaiian homology group as follows.

\begin{definition}\label{de1.3n}
Let $(X, x_0)$ be a pointed topological space. Then the $n$th  Hawaiian homology group of $(X, x_0)$, denoted by $\mathbb{H}_n(X, x_0)$, equals the quotient group $\frac{Z_n(X, x_0)}{B_n(X, x_0)}$, where $Z_n(X, x_0) = \ker \partial_n$ and $B_n(X, x_0) = \mathrm{im} \ \partial_{ n +1}$. By $\cls \alpha = \alpha + B_n(X, x_0)$ we mean the conjugacy class of $\alpha$.
\end{definition}

As a simple example, we compute the first Hawaiian homology group of the unit circle.

\begin{example}\label{ex1.4nn}
Let $\mathbb{S}^1$ denote the circle with radius $1$ in the Euclidean space. Then $\mathbb{H}_1(\mathbb{S}^1, a) = \frac{Z_n(\mathbb{S}^1, a)}{B_n(\mathbb{S}^1, a)}$.
Define $\psi :Z_n (\mathbb{S}^1, a) \to \sum_{\aleph_0} H_1(\mathbb{S}^1)$ by the rule $\psi (\sigma^n) = \{\cls \sigma^n|_{\Delta_k^n}\}_{k \in \mathbb{N}}$.
Since simplexes $\{\sigma^n|_{\Delta_k^n}\}_{k \in \mathbb{N}}$ tend to be small, they are contained in a contractible subset, and then they may be replaced with constant simplexes. Hence just a finite number of non-null-homotopic simplexes remain. Therefore, $\mathbb{H}_1(\mathbb{S}^1, a) \cong \sum_{\aleph_0} H_1(\mathbb{S}^1) \cong \sum_{\aleph_0} \mathbb{Z}$.
\end{example}

By Definition \ref{de1.3n}, for each pointed topological space, we correspond a sequence of abelian  groups. To define a functor $\mathbb{H}_n: Top_* \to Ab$, $n \ge 0$, from the category of pointed topological spaces to the category of abelian groups, we need to correspond morphisms; that is, for each pointed map $f:(X, x) \to (Y, y)$, a homomorphism $\mathbb{H}_n(f):\mathbb{H}_n(X, x) \to \mathbb{H}_n(Y, y)$ exists that satisfies
\begin{itemize}
\item[(i)]
$\mathbb{H}_n(id_X) = id_{\mathbb{H}_n(X, x)}$,
\item[(ii)]
$\mathbb{H}_n( f\circ g) = \mathbb{H}_n(f) \circ \mathbb{H}_n(g)$.
\end{itemize}
Similar to the definition of the functor of singular homology group \cite[p. 66]{Rot}, we define $\mathbb{H}_n(f):\mathbb{H}_n(X, x) \to \mathbb{H}_n(Y, y)$ by  $\cls z_n \mapsto \cls f_{\#} (z_n)$, where $f_{\#}: S_n(X, x) \to S_n(Y, y)$ is defined by $f_{\#} (\sum m_{\sigma}  \sigma) = \sum m_{\sigma} (f \circ \sigma)$. The same argument as used in \cite[Theorem 4.10]{Rot} implies that $\mathbb{H}_n(f)$ satisfies  conditions (i) and (ii), and then $\mathbb{H}_n$ is a functor.
Then it is natural to know if the Hawaiian homology functor preserves pointed homotopy mappings.
There are some differences between the singular homology group, as the most familiar homology theory, and the Hawaiian homology group. The first one is the preservation of homotopy mappings; for the singular homology group, it suffices to be a free homotopy for inducing an isomorphism,   but for the Hawaiian homology group, if the homotopy mapping does not preserve the base point, it does not necessarily induce an isomorphism (see Example \ref{ex1.4n}).

Now it is natural to know if the Hawaiian homology functors satisfy the homology axioms.

\subsection*{Axioms of homology}

In homology theory, there are five properties called axioms of homology or sometimes called Eilenberg-Steenrood axioms, which consists of dimension axiom, additivity, homotopy, exactness and excision.
These axioms classify different theories of homology and also help to calculate the homology group of some spaces such as $n$-spheres.
We see that the Hawaiian homology groups do not satisfy all of these axioms. In
 the following we investigate homology axioms for Hawaiian homology.
\begin{itemize}
\item
Dimension axiom for the Hawaiian homology group, similar to the singular homology group, can be verified directly by calculating boundary homomorphism; see \cite[Theorem 4.12]{Rot}.

\item
A modified version of the additivity axiom holds for Hawaiian homology group as stated in Theorem \ref{th2.8nn}. Moreover, Example \ref{ex2.9nn} shows that classical additivity axiom does not hold for Hawaiian homology.

\item
Homotopy axiom holds if the homotopy mapping preserves the base point (Proposition \ref{pr2.5nn}). Moreover, consider the cone over the $1$-dimensional Hawaiian earring. Its $1$st Hawaiian homology group is not trivial (See Theorem \ref{th2.10nn} and \cite[Theorem 2.13]{BabMas1}), while it is contractible.

\item
The exactness and excision axioms hold for the Hawaiian homology by the definition \ref{de1.3n} and the proof is straightforward.
\end{itemize}

%Note that
%Moreover, for the other axioms, we find some modified versions; for instance,  First, by an example we show that the common homotopy axiom does not hold for the Hawaiian homology group.

\begin{proposition}\label{pr2.5nn}
Let $f,g:(X, x_0) \to (Y, y_0)$ be pointed maps and let $F: f \simeq g$ rel $\{x_0\}$. Then $\mathbb{H}_n(f) = \mathbb{H}_n (g)$ for $n \ge 0$.
\end{proposition}

\begin{proof}
To prove the equality $\mathbb{H}_n(f) = \mathbb{H}_n(g)$, for every $n$-Hawaiian simplex $\alpha$, we show that   $\mathbb{H}_n(f) (\cls  \alpha) = \mathbb{H}_n(g) (\cls \alpha)$. It is equivalent to $f_{\#} (\alpha) - g_{\#} (\alpha)$ being the boundary of an $(n+1)$-Hawaiian simplex. By \cite[Theorem 4.23]{Rot}, since $F:f \simeq g$, $f^s_{\#} (\alpha|_{\Delta^n_k}) - g^s_{\#} (\alpha|_{\Delta^n_k})$ equals the boundary of an $(n+1)$-simplex in $Y$, namely $\beta_k$ is  constructed by $F(\alpha|_{\Delta^n_k}, -)$ for $k = 0, 1, 2, \ldots$.  Since $F$ is a pointed homotopy mapping, and $\mathbb{I}$ is compact,
%for every $t \in \mathbb{I}$, $F(x_0, t) y_0$. If $U$ be any open set containing $y_0$, then there is open interval $J_t \subseteq \mathbb{I}$, such that $F(x, s) \in U$, for $s \in J_t$ and $x \in V_t$, where $V_t$ is an open naborhood of $x_0$.  Unit interval $\mathbb{I}$ is covered by $\{J_t\}_{t \in \mathbb{I}}$. Let $\{t_1, \ldots, t_m\}$ be finite set such that $\mathbb{I} \subseteq \bigcup_{l = 1}^{m} J_t$. Put $V = \cap_{l = 1}^m V_{t_l}$. Then $F(x, s) \in U$ for all $x \in V$ and $s \in \mathbb{I}$.
the images of $\beta_k$'s tend to the point $y_0$ (for a detailed proof see \cite[Theorem 2.13]{BabMas1}). Then one can define the $(n+1)$-Hawaiian simplex $\beta = \cup_{k= 0}^{\infty} \beta_k$, where $\beta_0$ is the constant map at the point $y_0$. Moreover, since $f^s_{\#} (\alpha|_{\Delta^n_k}) - g^s_{\#} (\alpha|_{\Delta^n_k})$ for each $k= 0, 1, \ldots$, $f_{\#} (\alpha) - g_{\#} (\alpha)$ equals the boundary of $\beta$.
\end{proof}

A space $X$ is called semilocally strongly contractible at the point $x_0$ if there exists an open neighborhood $U$ of $x_0$ such that the inclusion map $i :U \to X$ is homotopic to the constant map relative to the point $x_0$. Since the image of any Hawaiian simplex is contained in $U$ except a finite number of standard simplexes and $U$ can be contracted to $x_0$, Hawaiian simplexes have no information more than finite standard simplexes. Example \ref{ex1.4nn} calculates the Hawaiian homology group of the circle $\mathbb{S}^1$, and it is formally generalized as follows.

\begin{corollary}\label{co2.6nn}
Let $X$ be  semilocally strongly contractible at the point $x_0$. Then
\[
\mathbb{H}_n(X, x_0) \cong \sum_{\aleph_0} H_n(X).
\]
\end{corollary}

\begin{proof}
We consider $H_n(X)$ as a subgroup of $\mathbb{H}_n(X, x_0)$, where each  simplex vanishes except the $k$th one corresponded to the elements of $H_n(X)$. The intersections vanish, because each factor is corresponded to a unique $k \in \mathbb{N}$. It remains to verify that $\mathbb{H}_1(X, x_0)$ is  generated by $\sum_{\aleph_0} H_n(X)$.
Let $\sigma$ be an $n$-Hawaiian simplex in $X$ at $x_0$. Since $X$ is semilocally strongly contractible at $x_0$, there exists a neighborhood $U$ of $x_0$ that is contractible in $X$ at $x_0$. Since each $n$-Hawaiian simplex is a union of null-convergent standard $n$-simplexes (see \cite{BabMas1}), there is $K \in \mathbb{N}$ such that $im (\sigma_k) \subseteq U$ for $k \ge K$. Then $\bigcup_{k \ge K}\sigma_k$ is null-homotopic in $X$, and thus $\mathbb{H}_n(X, x_0)$ is generated by $\sum_{\aleph_0} H_n(X)$.
\end{proof}

We present a space in Example \ref{ex1.4n} to prove that a free homotopy equivalence does not necessarily induce the  isomorphism on Hawaiian homology groups. %Now we see that  pointed homotopy mappings induce isomorphisms by Hawaiian homology functor.

\begin{example}\label{ex1.4n}
Let $C(\mathbb{HE}^1)$ be the cone over the one-dimensional Hawaiian earring. Since the cones are contractible, the homotopy axiom implies that the homology group must vanish, but it does not hold for the first Hawaiian homology group. To prove this, consider the simple $1$-Hawaiian simplex $\sigma$ whose image is the lowest Hawaiian earring. This Hawaiian simplex is contained in $Z_1(C\mathbb{HE}^1, \theta) = \ker \partial_1$, and also there is no $2$-Hawaiian simplex whose boundary equals $\sigma$.
\end{example}

\begin{theorem}\label{th2.8nn}
Let $X$ be semilocally path-connected at point $x_0$, let $\{X_{\lambda}\}$ be the family of path components with $x_0 \in X_{\lambda_0}$, and let $n \ge 0$. Then
\begin{equation}\label{eq1.1n}
\mathbb{H}_n (X, x_0) \cong \big(\sum_{\lambda \neq \lambda_0} \sum_{\aleph_0} H_n(X_{\lambda}) \big) \oplus \mathbb{H}_n(X_{\lambda_0} , x_0).
\end{equation}
\end{theorem}

\begin{proof}
First,  note that for each finite sequence of standard simplexes, one may construct a Hawaiian simplex by vanishing all terms of the sequence except the finite number of given simplexes. Therefore, $\sum_{\aleph_0} H_n(X_{\lambda})$ can be considered as a natural subgroup of $\mathbb{H}_n (X, x_0)$.
Moreover, consider two elements $\cls \sigma, \cls \beta  \in \mathbb{H}_n(X_{\lambda_0} , x_0)$. If $\cls \sigma = \cls \beta$, then $\sigma- \beta$ equals the boundary   of some $(n+1)$-Hawaiian simplexes in $X_{\lambda_0}$. Thus it holds for $\sigma- \beta$ as an $n$-Hawaiian simplex in $X$. Then $\cls \sigma = \cls \beta$ in $\mathbb{H}_n(X, x_0)$. Hence, $\mathbb{H}_n(X_{\lambda_0} , x_0)$  is a subgroup of $\mathcal{H}_n(X, x_0)$. Also, since $X$ is semilocally path-connected, there exists a path-connected open neighborhood, namely $U$ of $X$ at $x_0$.
%Since $U$ is path-connected and $X_{\lambda_0}$ is the path component containing $x_0$, $U \subseteq X_{\lambda_0}$.
We show that the abelian group $\mathbb{H}_n (X, x_0)$ is generated by the family of subgroups  $\{\sum_{\aleph_0} H_n(X_{\lambda})\}_{\lambda \neq \lambda_0} \cup \{  \mathbb{H}_n(X_{\lambda_0} , x_0)\}$. To prove this, consider a Hawaiian simplex $\sigma$. Then there is a natural number $K$ such that if $k \ge K$, then $im (\sigma (\Delta_k^n)) \subseteq U$. Note that since $U$ is path-connected, $U \subseteq X_{\lambda_0}$, and then $im(\sigma|_{\bigcup_{k \ge K} \Delta_k^n}) \subseteq X_{\lambda_0}$. Therefore,  $\sigma|_{\bigcup_{k \ge K} \Delta_k^n}$ can be considered as an element of $\mathbb{H}_n(X_{\lambda_0}, x_0)$ and the other standard simplexes $\Delta_k^n$, $k <K$, are generated by   $\sum_{\lambda \neq \lambda_0} \sum_{\aleph_0} H_n(X_{\lambda})$, as desired.
Since these subgroups have  trivial intersections, because of the path-connectedness of simplexes, the isomorphism holds.
\end{proof}

%\begin{proof}
%Define the map $\psi: \mathbb{H}_n (X, x_0) \to \sum_{\lambda \neq \lambda_0} \sum_{\aleph_0} H_n(X_{\lambda}) \oplus \mathbb{H}_n(X_{\lambda_0} , x_0)$ by the rule $\psi (\cls  \alpha) = ( \{\{\cls \alpha|_{\Delta_k^n}\}_{k \in \mathbb{N}}\}_{\lambda}, \alpha$
%\end{proof}

Now, we present an example to show that the isomorphism \eqref{eq1.1n} differs from the similar statement for the singular homology,
\[
H_n(X) \cong \sum_{\lambda} H_n(X_{\lambda}).
\]

\begin{example}\label{ex2.9nn}
Consider the space $X = \mathbb{HE}^2 \dot{\cup} \mathbb{S}^2$ and the point $a \in \mathbb{S}^2$. Then $\mathbb{H}_2( X, a) \cong \sum_{\aleph_0} H_2(\mathbb{HE}^2) \oplus \mathbb{H}_2( \mathbb{S}^2, a) = \sum_{\aleph_0} \prod_{\aleph_0} \mathbb{Z} \oplus \sum_{\aleph_0} \mathbb{Z}$ \cite{EdaKaw}. On the other hand, $\mathbb{H}_2(\mathbb{HE}^2, \theta) \oplus \mathbb{H}_2 (\mathbb{S}^2, a) \cong \prod_{\aleph_0} \sum_{\aleph_0} \mathbb{Z} \oplus \sum_{\aleph_0} \mathbb{Z}$ and these two groups are not isomorphic \cite{Fuc}.
\end{example}

By a well-known fact called Hurewicz theorem, for path-connected spaces, two functors of homotopy and homology are  connected. In fact, for a path-connected space $X$, there exists a surjective homomorphism $\phi_s: \pi_1(X, x_0) \to H_1(X)$ whose kernel equals the commutator subgroup of $\pi_1(X, x_0)$.
A similar relation exists between Hawaiian homology and Hawaiian groups of pointed spaces. For each pointed space $(X, x_0)$, there exists a homomorphism $\phi: \mathcal{H}_1(X, x_0) \to \mathbb{H}_1(X, x_0)$. The homomorphism $\phi$ is surjective and its kernel equals the commutator subgroup if $X$ is path-connected and locally path-connected at $x_0$.

\begin{theorem}\label{th2.10nn}
Let $X$ be a path-connected space that is locally path-connected at $x_0$. Then the homomorphism $\phi : \mathcal{H}_1(X, x_0) \to \mathbb{H}_1(X, x_0)$ is surjective whose kernel equals the commutator subgroup of $\mathcal{H}_1 (X, x_0)$.
\end{theorem}

\begin{proof}
First, we correspond an element of $\mathbb{H}_1(X, x_0)$ for each map $f: (\mathbb{HE}^1, \theta) \to (X, x_0)$. Let $\eta: \bigcup_{k =0}^{\infty} \Delta_k^1 \to \mathbb{HE}^1$ maps each $\Delta_k^1$ onto $\mathbb{S}_k^1$ homeomorphically except at the vertices mapped to the base point $\theta$. Then $f \eta$ is a $1$-Hawaiian simplex in $X$ at $x_0$. Also $\partial (f\eta) = \bigcup_{k =0}^{\infty} f_k \eta_k (e_1^k) - \bigcup_{k =0}^{\infty} f_k \eta_k (e_0^k) = \bigcup_{k =0}^{\infty} \big( f_k \eta_k (e_1^k) - f_k \eta_k (e_0^k) \big) = \bigcup_{k =0}^{\infty} = 0$, where $\eta_k :\Delta_k \to \mathbb{S}_k^1$ is the restriction of $\eta$ to the $k$th standard $1$-simplex. Thus $f \eta \in im(\partial_1)$. Hence one can define $\phi:\mathcal{H}_1 (X, x_0) \to \mathbb{H}_1(X , x_0)$ by  $\phi ([f]) = cls f \eta$.

 To prove well-definedness,   let $f \simeq g\ rel \{\theta\}$. Then $f_k \simeq g_k \ rel \{\theta\}$ for each $k =0, 1, \ldots$, where $f_k$ and $g_k$ are the restrictions of $f$ and $g$, respectively, to the $k$th circle $\mathbb{S}_k^1$. By the standard Hurewicz theorem, $cls f_k \eta_k = cls g_k \eta_k$ for $k = 0, 1, \ldots$. Thus $cls (f_k \eta_k - g_k \eta_k) =0$, and then $f_k \eta_k - g_k \eta_k$ belongs to the image of $\partial$, the  boundary of some linear composition of  standard $2$-simplex $\sum_{j=1}^{n} m_j\sigma^j_k$. By the proof of the standard Hurewicz theorem \cite[Lemma 4.26]{Rot}, this standard $2$-simplex is constructed by the homotopy mapping $H_k: f_ k \simeq g_k$ for $k =1 ,2 , \ldots$. Since $H_k$ is the restriction of the homotopy $H:f \simeq g$ to the $k$th circle, the image of $H_k$ is null-convergent, that is, $im (H_k) \subseteq  U$ if $k \ge K$ for some $K \in \mathbb{N}$ corresponded to the given open set $U$ containing $x_0$ (see \cite[Definition 2.1]{BabMas1}). Therefore, $im (\cup_{j =0}^n\sigma^j_k)$ is null-convergent and then $\cup_{k =0}^{\infty} \sum_{j =1}^nm_n  \sigma^j_k$, where the $0$-simplex is the constant standard simplex, is a linear composition of $2$-Hawaiian simplexes, whose boundary equals $f \eta - g \eta$. Hence $cls\ f-g =0$, and then $cls\ f = cls\ g$. It implies that $\phi$ is well-defined.

  Moreover, $\phi$ is a homomorphism. To prove this, consider two elements $[f]$ and $[g]$ in the group $\mathcal{H}_1(X, x_0)$. Then the  following sequence of equalities holds:
\begin{align*}
\phi ([f] [g]) = \phi ([f*g])  = cls \ (f*g) \eta = cls\ \cup_{k =0}^{\infty} (f_k * g_k) \eta_k.
\end{align*}
We must verify the equality
\[cls \ \cup_{k =0}^{\infty} (f_k * g_k) \eta_k = cls \ \cup_{k =0}^{\infty} f_k \eta_k + cls \ \cup_{k =0}^{\infty} g_k \eta_k.
\]
Equivalently, we must find a composition of  $1$-Hawaiian simplexes whose boundary equals $\cup_{k =0}^{\infty} (f_k * g_k) \eta_k - \cup_{k =0}^{\infty} f_k \eta_k -  \cup_{k =0}^{\infty} g_k \eta_k$. By the proof of the standard Hurewicz theorem \cite[Theorem 4.27]{Rot}, there exists a sequence of composition of standard $2$-simplexes, namely $\sum_{j =1}^n m_n \sigma^j_k$ ($k \in \mathbb{N}$) constructed by the maps $f_k*g_k$, $f_k$, and $g_k$ such that $\partial (\sum_{j =1}^n m_n \sigma^j_k) = (f_k *g_k)\eta_k - f_k \eta_k -  g_k \eta_k$. It remains to verify that the sequence $\{\sum_{j =1}^n m_n \sigma^j_k\}$ is null-convergent. Since the maps $f_k *g_k$, $f_k$, and $g_k$ are null-convergent, so is $\sum_{j =1}^n m_n \sigma^j_k$ ($k \in \mathbb{N}$). Therefore, $\cup_{k = 0}^{\infty} \sum_{j =1}^n m_n \sigma^j_k$ is a $2$-Hawaiian simplex, and also its boundary equals $\cup_{k =0}^{\infty} (f_k * g_k) \eta_k -  \cup_{k =0}^{\infty} f_k \eta_k -  \cup_{k =0}^{\infty} g_k \eta_k = (f*g) \eta - f \eta - g \eta$. Thus $cls \ (f*g) \eta = \cls\ f \eta + cls \  g \eta$, and then $\phi$ is a homomorphism.

To prove surjectivity of the homomorphism $\phi$, we consider an element $cls\ \sigma$ of $\mathbb{H}_1(x, x_0)$. By the  restriction to the $k$th $1$-simplex, denoted by $\sigma_k$, one can use the standard Hurewicz theorem \cite[Theorem 4.29]{Rot}, to obtain a sequence of maps $f_k$ defined from $\mathbb{S}^1$ to $X$ such that $\phi_s([f_k]) =cls\ \sigma_k$. If $\sigma_k = \sum_{i =0}^m \sigma_k^i$, then by the proof of \cite[Theorem 4.29]{Rot}, $f_k$ is defined by the concatenation of a  sequence of simplexes and some paths between the base point $x_0$ and the end points of simplexes.
Since $X$ is locally path-connected at $x_0$ having a countable local basis, one can consider a  sequence of nested open neighborhoods $\{U_i\}_{i \in \mathbb{N}}$ of $x_0$ such that each element $U_i$ is path-connected  in $U_{i+1}$. Therefore, the sequence $f_k$ can be considered to be null-convergent.
 Then define $f|_{\mathbb{S}_k^1} = f_k$. Since $\{f_k\}$ is null-convergent, $f$ is continuous by \cite[Lemma 2.2]{BabMas1}, and also $\phi([f]) = \sigma$.
To calculate the kernel of $\phi$, again use the proof of \cite[Theorem 4.29]{Rot}, \cite[Lemma 2.2]{BabMas1}, and locally path-connectedness at $x_0$.
%
%Similar to the proof of well-definedness, some sequence of compositions standard $2$-simplexes exist whose boundary equals
%The homomorphism $\phi$ is surjective. To prove this, let $\cls \sum_{j = 1}^n m_n\sigma^j$ be an element of $\mathbb{H}_1 (X, x_0)$. Put $\sigma^j_k := \sigma^j|_{\Delta_k}$. Since $X$ is path-connected, for each $k \in \mathbb{N}$,  there exists a map $f_k: (\mathbb{HE}^1, \theta) \to (X, x_0)$ such that $\phi_s ([f_k]) \sum_{j =1}^n m_n \sigma_k^j$
\end{proof}

By Theorem \ref{th2.10nn} and the first isomorphism theorem, it holds that
\[
\mathbb{H}_1(X, x_0) \cong \frac{\mathcal{H}_1(X, x_0)}{\big(\mathcal{H}_1(x, x_0)\big)'},
\]
where $\big(\mathcal{H}_1(x, x_0)\big)'$ denotes the commutator subgroup of the given group $\big(\mathcal{H}_1(x, x_0)\big)$.
Then the first Hawaiian homology group  is isomorphic to the abelianization of the first Hawaiian group for path-connected and locally path-connected pointed topological spaces. This fact helps us to study Hawaiian groups of topological spaces by Hawaiian homology.

%\begin{comment}
\section{Higmann-complete Hawaiian groups}

Herfort and Hojka proved that the fundamental group of archipelago spaces is Higman-complete, introduced as a non-abelian form of cotorsion groups; for more details, see   \cite[Definition 1]{HerHoj}. In the rest of the paper, all the spaces are assumed to be path-connected.

\begin{definition}[\cite{HerHoj}]
A group $G$ is Higman-complete if for any sequence
$f_1, f_2, \ldots \in G$ and for a given sequence of words $w_1, w_2, \ldots$, there exists a
sequence $h_1, h_2, \ldots \in G$ such that all equations
$h_i = w_i(f_i, h_{i+1})$ hold simultaneously.
\end{definition}

The following lemma states that the first Hawaiian  group of any space whose loops are small, belongs to the class of Higman-complete groups. This fact also holds for the fundamental group of archipelago spaces \cite{HerHoj}.

\begin{lemma}\label{le0.1}
If all $1$-loops at the point $x \in X$ are small, then $\mathcal{H}_1(X, x)$ is Higmann-complete.  If it is true for $n$-loops, $n \ge 2$, then $\mathcal{H}_n(X, x)$ is cotorsion.
\end{lemma}

\begin{proof}
Let all $1$-loops be small at $x \in X$. Assume that $f: \mathbb{HE}^1 \to X$ is a $1$-Hawaiian loop at $x$. First, we show that $f$ is a small $1$-Hawaiian loop. That is for each open neighborhood $U$ of $x$, there is a homotopic representative $g \simeq f$ with image contained in $U$. Since $f$ is continuous, there is $K \in \mathbb{N}$ such that $im(f|_{\mathbb{S}_k^1}) \subseteq U$ for $k \ge K$. Moreover, by the assumption, all $f|_{\mathbb{S}_k^1}$'s are small. For $k <K$, define $g_k$ as the homotopic representative $1$-loop of $f|_{\mathbb{S}_k^1}$ with image in $U$. Now put $g:\mathbb{HE}^1 \to X$ as $g|_{\mathbb{S}_k^1} = g_k$ for $k <K$, and $g|_{\mathbb{S}_k^1} = f|_{\mathbb{S}_k^1}$ for $k \ge K$, which is homotopic to $f$ by \cite[Lemma 2.2]{BabMas1}. In \cite[Theorem 4]{HerHoj}, it was shown that $\pi_1(X, x)$ is Higmann-complete, where all $1$-loops at $x$ are small.  Now by a similar argument, one can prove that $\mathcal{H}_1(X, x)$ is Higmann-complete.  Similarly, if all $n$-loops are small, $n \ge 2$, then $\pi_n(X, x)$ and $\mathcal{H}_n(X, x)$ are Higmann-complete. Since $\pi_n(X, x)$ and $\mathcal{H}_n(X, x)$ are abelian, by \cite[Theorem 3]{HerHoj},  they are cotorsion.
\end{proof}

Theorem \ref{th2.10nn} presents a version of Hurewicz theorem for Hawaiian groups and Hawaiian homology groups. Now we use that theorem to obtain the following corollary.

\begin{corollary}
Let $X$ be locally path-connected at $x$, and let all $1$-loops at $x$ be small. Then the first Hawaiian homology group $\mathbb{H}_1(X, x)$ is cotorsion.
\end{corollary}

\begin{proof}
Since all $1$-loops at $x$ are small, by Theorem \ref{le0.1}, $\mathcal{H}_1(X, x)$ is Higmann-complete. Also by Theorem \ref{th2.10nn}, the group $\mathbb{H}_1(X, x)$ is the epimorphic image of the group $\mathcal{H}_1(X, x)$.  Then it is Higmann-complete by \cite[Lemma 2]{HerHoj}. Moreover, every abelian group is Higmann-complete if and only if it is cotorsion \cite[Theorem 3]{HerHoj}. Therefore $\mathbb{H}_1(X, x)$ is cotorsion.
\end{proof}

If all $n$-loops at a point are small, then both $n$th homotopy and $n$th Hawaiian groups are Higman-complete. Note that the homotopy groups on path-connected spaces do not depend on the choice of the base point. It implies that $\pi_n(X, x)$ is Higmann-complete if all $n$-loops at some points of $X$ are small, but it is not true for Hawaiian groups. Some counterexamples are investigated in Corollary \ref{co3.11nn}, by applying the following lemma.

\begin{lemma}\label{le0.2}
Let $X$ be semi-locally strongly contractible at $x_0$ and let $n \ge 2$. If $H_1(X, x_0)$ is not torsion, then neither $\mathbb{H}_1(X, x_0)$ is  cotorsion, nor  $\mathcal{H}_1(X, x_0)$ is Higman-complete. Moreover, if $\pi_n(X, x_0)$ is not torsion, then $\mathcal{H}_n(X, x_0)$ is not cotorsion.
\end{lemma}

\begin{proof}
Since $X$ is semi-locally strongly contractible at $x$, by Corollary \ref{co2.6nn}, the isomorphism $\mathbb{H}_1(X, x) \cong \sum_{\aleph_0} H_1(X)$ holds. If $\mathbb{H}_1(X, x)$ is cotorsion, then $\sum_{\aleph_0} H_1(X)$ is cotorsion. Thus by \cite[Proposition 6.10]{Fuc2015}, $H_1(X)$ is torsion, which is a contradiction.
By \cite[Theorem 2.5]{BabMas1}, $\mathcal{H}_1(X, x_0) \cong \prod^W_{\aleph_0} \pi_1(X, x_0)$. If $\mathcal{H}_1(X, x_0)$ is Higman-complete, then $\prod^W_{\aleph_0} \pi_1(X, x_0)$ is Higmann-complete.
Thus by \cite[Theorem 3]{HerHoj}, its abelianization $\mathbb{H}_1(X, x) \cong \sum_{\aleph_0} H_1(X)$ is cotorsion, which does not hold.
%
% as its epimorphic image by \cite[Lemma 2]{HerHoj}. Since $X$ is path-connected, by Hurewicz theorem $H_1(X)$ is the epimorphic image of the fundamental group $\pi_1(X, x)$, and then it is Higmann-complete. Since $H_1(X)$ is abelian and Higmann-complete, it
% is cotorsion by \cite[Theorem 3]{HerHoj}. We know that $Ab(G) \cong \sum_{\aleph_0} Ab(\pi_1(X, x_0)) \cong \sum_{\aleph_0} H_1(X)$.
%For $n \ge 2$, if $\mathcal{H}_n(X, x_0) \cong \sum_{\aleph_0} \pi_n(X, x_0)$ is cotorsion, then  by \cite[Proposition 6.10]{Fuc2015}, $\pi_n(X, x_0)$ is torsion.
\end{proof}

CW spaces are semilocally strongly contractible at any point, and then by Lemma \ref{le0.2}, the following corollary is obtained.

\begin{corollary}\label{co0.3}
Let $X$ be a CW space with torsion-free first homology group. Then $\mathcal{H}_1(X, x)$ is not Higman-complete for all $x \in X$. Moreover, if $H_n(X)$ is torsion-free, then $\mathbb{H}_n(X, x)$ is not cotorsion for $x \in X$.
\end{corollary}

If $\mathcal{H}_n(X, x)$ is Higman-complete for some $x \in X$, then so is $\pi_n(X, x)$ as its epimorphic image, but the converse statement does not hold. There are spaces such that their homotopy groups are Higman-complete but not their Hawaiian groups neither Hawaiian homology groups; see the following example.

\begin{example}\label{ex0.7}
Consider a Higman-complete torsion-free  group, namely $\mathbb{Q}$, and $X = K(\mathbb{Q}, 1)$ as the Eilenberg--MacLane space corresponded to $\mathbb{Q}$. Then $H_1(X) \cong \pi_1(X) \cong \mathbb{Q}$, which is Higman-complete because it is cotorsion. Moreover, since $X$ is semi-locally strongly contractible at any point $x$, $\mathcal{H}_1(X, x) \cong \sum_{\aleph_0} \pi_1 (X, x) \cong \sum_{\aleph_0} \mathbb{Q}$ by \cite[Theorem 2.5]{BabMas1}. Since $X$ is a  CW space and $H_1(X)$ is torsion-free, by Corollary \ref{co0.3}, $\mathcal{H}_1(X, x)$ is not Higman-complete for each $x \in X$. Also, $\mathbb{H}_1(X, x) \cong \sum_{\aleph_0} H_1(X) \cong \sum_{\aleph_0} \mathbb{Q}$  by Corollary \ref{co2.6nn}, and then $\mathbb{H}_1(X, x)$ is not cotorsion for each $x \in X$ by Corollary \ref{co0.3}.
%If $\mathcal{H}_1(X, x)$ is Higman-complete, then it is cotorsion because it is abelian. Now  by \cite[Proposition 6.10]{Fuc2015}, its direct summand being isomorphic to $\mathbb{Q}$ should be torsion which is a contradiction.
\end{example}

Some pseudomanifolds satisfy the conditions of Corollary \ref{co0.3}; see \cite[Theorem 8.2]{Mas}.

\begin{corollary}
The first Hawaiian group of an orientable two-dimensional pseudomanifold $X$ is not Higman-complete. Also its first Hawaiian homology group $\mathbb{H}_1(X, x)$ is not cotorsion for all $x \in X$.
\end{corollary}

Archipelago groups are the first examples of Higman-complete groups. They are introduced as the fundamental groups of archipelago spaces.

The Harmonic archipelago is a non-locally Euclidean space, which was introduced to be a counterexample for some statements. This space was generalized  to archipelago space \cite{ConHoj}. Herfort and  Hojka \cite{HerHoj} computed the singular homology groups of archipelago spaces. In this paper, we use a similar argument to present Hawaiian homology groups of the archipelago spaces.

%One can verify that if all $n$-loops at the point $x$ are small, then so are $n$-Hawaiian loops. Using this fact, the following lemma can be proved   by the same argument as \cite[Theorem 4]{HerHoj}.

\begin{definition}[\cite{ConHoj}]
Let $\{(X_i,x_i)\}_{i\in I}$ be a family of pointed spaces. Then the archipelago space  $\mathbb{A} = \mathbb{A}(\{X_i\}_{i \in I})$ on the given family  is defined as the mapping cone on the natural continuous bijection $f: \bigvee_{i \in I} (X_i, x_i) \to \widetilde{\bigvee}_{i \in I} (X_i, x_i)$. That is, the quotient space $C_f$ is obtained from $(X \times \mathbb{I}) \cup Y$ by identifying $(x, 1/2)$ with $f(x)$ and $X \times \{1\}$ is contracted to a point.
\end{definition}

 Archipelago spaces are examples of wild spaces. Free $\sigma$-groups are also fundamental groups of some wild spaces \cite[Theorem A.1]{Eda}. Indeed free $\sigma$-groups are not Higman-complete with behaves similar to free abelian groups,  which are not cotorsion.

\begin{remark}
If $\{G_i\}_{I}$ is a family of groups such that $Ab(G_i)$ is not cotorsion for some $i \in I$, then the free product $\ast_{I} G_i$ and free $\sigma$-product $\circledast_I G_i$ are not Higman-complete.
In fact, there are natural epimorphisms $\circledast_I G_i \to \prod_I G_i$ and $\ast_I G_i \to \prod^W_I G_i$. Then  $G_i$ and hence $Ab(G_i)$ are their epimorphic images for all $i \in I$. If any of $\circledast_I G_i$ or $\ast_I G_i$ is Higman-complete, then so is $G_i$ for all $i \in I$ by \cite[Lemma 2]{HerHoj}. Thus $Ab(G_i)$ is cotorsion for all $i \in I$ by \cite[Theorem 3]{HerHoj}.
\end{remark}

By Lemma \ref{le0.2}, Hawaiian groups of semi-locally strongly contractible spaces are not Higman-complete. In the following result, we see that there are spaces, such as the one-dimensional Hawaiian earring, that are not semi-locally strongly contractible and their Hawaiian groups are Higman-complete.

\begin{theorem}\label{th0.3}
Let $\{X_i\}_{i \in I}$ be a family of connected spaces, let $X = \widetilde{\bigvee}_I X_i$, and  let $x$ be the common point. If $H_1(X_i)$ is not cotorsion for some $i \in I$, then $\mathbb{H}_1(X, x)$ is not cotorsion, and moreover, $\pi_1(X, x)$ and $\mathcal{H}_1(X, x)$ are not Higman-complete.
\end{theorem}

\begin{proof}
Since $X_i$ is a retract of the space $X$, there exits a natural epimorphism $\mathbb{H}_1(X, x) \to \mathbb{H}_1(X_i, x_i)$ induced by the Hawaiian homology functor. Also
By \cite[Theorem A.1]{Eda}, $\pi_1(X, x) \cong \circledast_I G_i$, where $G_i = \pi_1(X_i, x_i)$.  If $\mathcal{H}_1(X, x)$ is Higman-complete, then so is its epimorphic image $H_1(X_i)$ by \cite[Lemma 2]{HerHoj} for all $i \in I$.  Thus, since $H_1(X_i)$ is abelian by \cite[Theorem 3]{HerHoj}, it must be cotorsion for all $i \in I$.
\end{proof}

Theorem \ref{th0.3} and Lemma \ref{le0.1} imply that  one of the groups in the following corollary is Higman-complete but not the other one. One can generalize Griffiths space $\mathcal{G}(X_i)$ over an arbitrary family of spaces $\{X_i\}_{\mathbb{N}}$ as the wedge of two cones on $\widetilde{\bigvee}_{\mathbb{N}} X_i$.

\begin{corollary}\label{co3.11nn}
Let $\{X_i\}_{i \in I}$ be as in Theorem \ref{th0.3}, let $X = \widetilde{\bigvee}_{i \in I}$ be the week join space, let $A(X)$ be the archipelago space, and let $\mathcal{G}(X)$ be the Griffiths space on $X$. Then $\mathcal{H}_1(A(X), x_*) \not \cong \mathcal{H}_1(X, x_*)$, and also, $\mathcal{H}_1( \mathcal{G}(X), x_*) \not \cong \mathcal{H}_1(X, x_*)$.
\end{corollary}

In the following theorem, we prove that for two points of the archipelago spaces, the Hawaiian groups are not isomorphic. Thus archipelago spaces are other examples of path-connected spaces whose Hawaiian groups are not isomorphic at different points.

\begin{theorem}\label{th0.7}
Let $\mathbb{A} = \mathbb{A}_{i \in I}(X_i)$ be the archipelago space over the family $\{(X_i, x_i)\}$. If $a \in \mathbb{A}$ is any point other than common point $x$, then  $\mathcal{H}_1(\mathbb{A}, x) \not \cong \mathcal{H}_1( \mathbb{A}, a)$.
\end{theorem}

\begin{proof}
By \cite[Theorem 8]{HerHoj}, $H_1( \mathbb{A}) \cong \frac{\prod_{\aleph_0} \mathbb{Z}}{\sum_{\aleph_0} \mathbb{Z}}$, which is torsion-free by \cite{Bal}. Since $\mathbb{A}$ is semi-locally strongly contractible at $a$,  by Lemma \ref{le0.2},  $\mathcal{H}_1( \mathbb{A}, a)$ is not Higman-complete. Also, since all $1$-loops at $x$ are small (see \cite[Lemma 7]{HerHoj}), by Lemma \ref{le0.1}  $\mathcal{H}_1(X, x)$ is Higman-complete, and then it is not isomorphic to $\mathcal{H}_1( \mathbb{A}, a)$.
%then If $G = \prod^W_{\aleph_0} \pi_1( \mathbb{HA}, a)$ is Higmann-complete, then so is $Ab(G)$ by \cite[Lemma 2]{HerHoj}. That is $\sum_{\aleph_0} Ab(\pi_1( \mathbb{HA}, a)) \cong \sum_{\aleph_0} H_1( \mathbb{HA})$ is Higmann-complete. By \cite[Theorem 3]{HerHoj}, $\sum_{\aleph_0} H_1( \mathbb{HA})$ is cotorsion. Hence, by \cite[Proposition 6.10]{Fuc2015}, $H_1( \mathbb{HA})$ is torsion.

\end{proof}

 Since the Griffiths space is not necessarily semi-locally strongly contractible at $a$, Lemma \ref{le0.2} cannot be used.

\begin{theorem}\label{th0.8}
If $a \in \mathcal{G} = \mathcal{G}(\{X_i\})$ is any point other than common point $x$, then  $\mathcal{H}_1(\mathcal{G}, x) \not \cong \mathcal{H}_1( \mathcal{G}, a)$.
\end{theorem}

\begin{proof}
By \cite{Eda91}, $\mathcal{G}$ satisfies the premises of Lemma \ref{le0.1}, and then $\mathcal{H}_1(\mathcal{G}, x)$ is Higman-complete.
%If $\mathcal{H}_1(\mathcal{G}, x) \not \cong \mathcal{H}_1( \mathcal{G}, a)$, then is Higman-complete too.
If $a \neq x_*$, then by \cite[Theorem 3.2]{BabMas2}, $\mathcal{H}_1( \mathcal{G}, a)$ has $\prod^W_{\aleph_0} \pi_1(\mathcal{G}, a)$ as a direct factor.  Therefore $\sum_{\aleph_0} H_1(\mathcal{G})$ is the epimorphic image of $\mathcal{H}_1(\mathcal{G}, a)$, and by \cite[Lemma 2 and Theorem 3]{HerHoj}, it is cotorsion. By \cite[Proposition 6.10]{Fuc2015}, $H_1(\mathcal{G})$ should be torsion,  which contradicts with \cite[p. 2]{KarRep3}.
\end{proof}

For both archipelago and Griffiths spaces, by the same argument as Theorems \ref{th0.7} and \ref{th0.8}, we can compare $L_1(\mathbb{A}(X_i))$ and $L_1(\mathcal{G}(X_i))$ at different points.  Then $L_1$ groups of both of archipelago and Griffiths spaces are not isomorphic at different points while they are path-connected spaces. It means that usual paths do not transfer Hawaiian groups and $L_n$ groups isomorphically.
%
%
%
%\begin{definition}
%Let $X$ be a space, $x_0, x_1 \in X$ and $1 \le n < \infty$. Let $\gamma$ be a path from $x_0$ to $x_1$, such that $\gamma^{-1}$ is $n$-SLT. Then, for any two $ n $-loops $\alpha, \alpha': (\mathbb{I}^n, \dot{\mathbb{I}^n}) \to (U, x_0) $, there are $n$-loops $\beta, \beta' :(\mathbb{I}^n, \dot{\mathbb{I}^n}) \to (V, x_1) $, which are homotopic to $\gamma_{\#}(\alpha)$ and $\gamma_{\#} (\alpha')$, respectively.
%We say that a path $\gamma$ from $x_0$ to $x_1$, with $n$-SLT inverse $\gamma^{-1}$, is small  homotopic $ n $-loop transfer (abbreviated to $ n $-SHLT),  whenever  $\alpha \simeq \alpha'$ in $U$ implies that  $\beta \simeq \beta'$ in $V$.
%\end{definition}
%
%Let $X$ be a space, $x_0, x_1 \in X$, and $\gamma$ be an $n$-SHLT path from $x_0$ to $x_1$. Also, let $\{U_m\}_{m \in \mathbb{N}}$ and $\{V_m\}_{m \in \mathbb{N}}$ be local nested bases at $x_0$ and $x_1$, respectively. Define $\Gamma_{\gamma}: \mathcal{H}_n(X, x_0) \to \mathcal{H}_n(X, x_1)$, by $\Gamma_{\gamma} ( [f]) = [g]$, where $g|_{\mathbb{S}_k^n}$ is the $n$-loop $\alpha_k$ being homotopic to $\gamma_{\#} (f|_{\mathbb{S}_k^n})$, with $ im (\alpha_k) \subseteq V_{m}$ for $K_m \le k < K_{m +1}$, and $K_m \in \mathbb{N}$ is so that $im (f|_{\mathbb{S}_k^n}) \subseteq U_{m}$, when $K_{m} \le k < K_{m+1}$ ($m \in \mathbb{N}$).
In the following theorem, we present an equivalent condition for paths to transfer Hawaiian groups  isomorphically. Some other conditions, namely $n$-SHLT introduced in \cite{BabMas3}, are enough but not necessary. Recall that a path $\gamma$ from $x_0$ to $x_1$ is called small $ n $-loop transfer (abbreviated to $ n $-SLT),  if for every open neighborhood $U$ of $ x_0 $, there exists an open neighborhood $ V $ of $x_1$, such that for every $ n $-loop $\beta: (\mathbb{I}^n, \dot{\mathbb{I}^n}) \to (V, x_1) $, there is an $n$-loop $\alpha :(\mathbb{I}^n, \dot{\mathbb{I}^n}) \to (U, x_0) $ being homotopic to $\gamma_{\#}^{-1}(\beta)$, where $\gamma_{\#}$ induces the homomorphism transfers $n$th homotopy group isomorphically along the points of path $\gamma$ (see \cite[p. 381]{Spa}).

\begin{theorem}
Let $X$ have local nested bases at $x_0$ and $x_1$. A path $\gamma$ from $x_0$ to $x_1$ in $X$, induces the isomorphism $\Gamma_{\gamma}:\mathcal{H}_n(X, x_0) \to \mathcal{H}_n(X, x_1)$ if and only if $\gamma$ and $\gamma^{-1}$ are $n$-SLT paths preserving null-convergent homotopies of $n$-loops.
\end{theorem}

\begin{proof}
Let $\Gamma_{\gamma}$ induce the isomorphism.
First, we show that $\gamma$ and the inverse path $\gamma^{-1}$ are $n$-SLT paths. Let $\{U_m\}$ and $\{V_m\}$ be local nested bases at  $x_0$ and $x_1$, respectively. Assume that $U$ is an open neighborhood of $x_0$ such that for each $m \in \mathbb{N}$, there is an $n$-loop $\beta_m$ in $V_m$ at $x_1$ without any homotopic $n$-loop to  $\gamma_{\#}^{-1} (\beta_m)$ in $U$.
Define $\beta:(\mathbb{HE}^n, \theta) \to (X, x_1)$ by $\beta|_{\mathbb{S}_m^n} = \beta_m$. From \cite[Lemma 2.2]{BabMas1}, $\beta$ is continuous. Since $\Gamma_{\gamma}$ is an epimorphism, there is $\alpha:(\mathbb{HE}^n, \theta) \to (X, x_0)$ such that $\Gamma_{\gamma} ([\alpha]) = [\beta]$. Then $\beta_k \simeq \gamma_{\#} (\alpha|_{\mathbb{S}_k^n})$ or equivalently $\gamma_{\#}^{-1} (\beta|_{\mathbb{S}_k^n}) \simeq \alpha_k = \alpha|_{\mathbb{S}_k^n}$. Since $\alpha$ is continuous and $U$ is open, there is $K \in \mathbb{N}$ such that if $k \ge K$, then $im ( \alpha|_{\mathbb{S}_k^n}) \subseteq U$. Because of the definition of $\beta$, it is a contradiction. Therefore, $\gamma$ is an $n$-SLT path. Similarly since $\Gamma_{\gamma}^{-1}$ is an epimorphism, $\gamma^{-1}$ is an $n$-SLT path.
Therefore, for each open neighborhood $U$ of $x_0$, there is an  open neighborhood $V$ of $x_1$ such that for any two $n$-loops $\alpha$ and $ \alpha'$ in $U$, there are $n$-loops $\beta$ and $ \beta'$ in $V$ homotopic to $\gamma_{\#} (\alpha)$ and $ \gamma_{\#} (\alpha')$, respectively.
Let $\{H_k\}$ be a null-convergent sequence of homotopies of $n$-loops. That is, there is an increasing sequence of natural numbers, namely $\{K_m\}$, such that for $K_m \le k <K_{m +1}$, $im (H_k) \subseteq U_m$. Since $\gamma^{-1}$ is an $n$-SLT path, there are $n$-loops $\beta_k$ and $ \beta_k'$ with $im(\beta_k), im (\beta'_k) \subseteq V_m$ such that $\beta_k \simeq \gamma_{\#}(H_k(-, 0))$ and
$\beta_k' \simeq \gamma_{\#}(H_k(-, 1))$ whenever $K_m \le k <K_{m+1}$. Define $f, f':(\mathbb{HE}^n, \theta) \to (X, x_0)$  and  $f, g:(\mathbb{HE}^n, \theta) \to (X, x_1)$ by $f|_{\mathbb{S}_k^n} = H_k(-, 0)$, $f'|_{\mathbb{S}_k^n} = H_k(-, 1)$,  $g|_{\mathbb{S}_k^n} = \beta_k$, and $g'|_{\mathbb{S}_k^n} = \beta_k'$. Now since $\{H_k\}$ is null-convergent, we have $[f] = [f']$ by \cite[Lemma 2.2]{BabMas1}. Also since $\gamma$ induces a well-defined homomorphism, $\Gamma_{\gamma}([f]) = \Gamma_{\gamma} ([f'])$. Therefore, $[g] = [g']$, or equivalently there is $G:g\simeq g'$. Since $G:\mathbb{HE}^n \times \mathbb{I} \to X$ is continuous and $\mathbb{I}$ is compact, $G|_{\mathbb{S}_k^n \times \mathbb{I}}$ is null-convergent.  Then $\{\beta_k\}$ and $\{\beta_k'\}$ are homotopic by a null-convergent sequence of homotopies.
Similarly since $\gamma^{-1}$  induces the well-defined homomorphism $\Gamma_{\gamma}^{-1}$, one can show that $\gamma^{-1}$ preserves null-convergent homotopies of $n$-loops.
The proof of the converse statement is the same as \cite[Theorem 5.5]{MasMir}.
%
%
%
%
%To verify $\gamma$ is $n$-SHLT, it remains to check if $\alpha \simeq \alpha'$ in $U$, then $\beta \simeq \beta'$ in $V$. For each $n$-loop $\alpha$, define $\overline{\alpha}$ as the composition of $\alpha$ with countably many nulhomotopic $n$-loops. Since $\mathcal{H}_n(U, x_0)$ is  a subgroup of $\mathcal{H}_n(X, x_0)$, the restriction $\eta = \Gamma_{\gamma}|_{\mathcal{H}_n(U, x_0)}$ is injective with image in $\mathcal{H}_n(V, x_1)$. If $\alpha \simeq \alpha'$ in $U$, then $\eta ([\overline{\alpha}]) = \eta ([\overline{\alpha'}])$. Since the restriction of $\eta$ is well-defined, $[\overline{\beta}] = \eta([\overline{\alpha}] = \eta([\overline{\alpha'}]= [\overline{\beta'}]$ in $\mathcal{H}_n(V, x_1)$ that implies that $\beta \simeq \beta'$ in $V$. Thus $\gamma$ is an $n$-SHLT path. Similarly since $\Gamma_{\gamma}^{-1}$ is well-defined, $\gamma^{-1}$ is an $n$-SHLT path as desired.
%
\end{proof}

\section{The first Hawaiian homology of Harmonic archipelago}

In this section, we obtain the structure of the first Hawaiian homology of the Harmonic archipelago spaces, up to isomorphism. In fact, the first Hawaiian homology group and the first singular homology group of the Harmonic archipelago spaces are isomorphic. First, we prove that it is locally free; that is, all of its finitely generated subgroups are free.

\begin{theorem}\label{th1.2}
Let $\mathbb{A} = \mathbb{A}(\{X_i\})$ be the archipelago space on $\{X_i\}_{i \in \mathbb{N}}$, where each $X_i$ is locally strongly contractible and let $\theta$ be the common point. If $\pi_1(X_i)$ is free for all $i \in I$ except a finite number, then $\mathcal{H}_1(\mathbb{A}, \theta)$ is locally free.
\end{theorem}

\begin{proof}
Since $X_i$ is locally strongly contractible and first countable at $x_i$, for $i\in\mathbb{N}$, there exists a nested local basis $\{V_j^i\}_{j\in\mathbb{N}}$ at $x_i$ such that the inclusion mapping $V_j^i\hookrightarrow V_{j-1}^i$ is null homotopic to $x_i$ in $V_{j -1}^i$. Let $\{U_m\}_{m\in\mathbb{N}}$ be the  local basis at $x_\ast$ obtained by $U_m=({\bigvee}_{ i<m}V_m^i)\vee \widetilde{\bigvee}_{i \ge m} X_i$. Therefore, the inclusion map $U_{m+1} \hookrightarrow U_m$ is homotopic to the map retracting each point onto $x_i$ for $i <m$ and identity for the others.
The topology of archipelago spaces implies that the natural basis at the common point can be retracted onto the neighborhoods $\{U_m\}$.
Consider $M$ as the minimum number such that $\pi_1(X_i)$ is free for all $i \ge M$. Let $G = \langle \{[f^1], [f^2], \ldots, [f^m]\}\rangle$ be a subgroup of $\mathcal{H}_1(\mathbb{A}, \theta)$. Also let $w$ be a word such that $w([f^1], [f^2], \ldots, [f^l])$ is the trivial element. Hence there exists a homotopy mapping $H : w(f^1, \ldots, f^l) \simeq C_{\theta}$.
%Let $U_m = \bigvee_{i <m} V_$ be the open set obtained by deleting all the picks of cells of the harmonic archipelago.
%Then  $X = \widetilde{\bigvee}_{i \in \mathbb{N}} X_i$ is a deformation retract of $U$ by a retraction namely $r : U \to X$. Therefore $\pi_1(r): \pi_1(U, \theta) \to  \pi_1(X)$ is an isomorphism.
Since $H$ is continuous, for each $m \in \mathbb{N}$, there exists $K_m \in \mathbb{N}$ such that $im(H|_{\mathbb{S}_k^1 \times \mathbb{I}})$ can be considered homotopically as a subset of $U_{m+1}$ if $k \ge K_m$. Then, $im(w(f|_{\mathbb{S}_k^1}^1, f|_{\mathbb{S}_k^1}^2, \ldots, f|_{\mathbb{S}_k^1}^l)) \subseteq U_m$  for $k \ge K_m$. Thus $r (w(f|_{\mathbb{S}_k^1}^1, f|_{\mathbb{S}_k^1}^2, \ldots, f|_{\mathbb{S}_k^1}^l))$ induces a word of $1$-loops at $\theta$ in $\widetilde{\bigvee}_{i \ge m} X_i$ for $k \ge K_m$, where $r$ is the natural retraction onto the space $\widetilde{\bigvee}_{i \ge m} X_i$. Also, $r\circ H|_{\mathbb{S}_k^1 \times \mathbb{I}} : r (w(f|_{\mathbb{S}_k^1}^1, f|_{\mathbb{S}_k^1}^2, \ldots, f|_{\mathbb{S}_k^1}^l)) \simeq c_{\theta}$ in $\widetilde{\bigvee}_{i \ge m} X_i$ for $k \ge K_m$. Since $\pi_1(\widetilde{\bigvee}_{i \ge m} X_i)$ is locally free, all $1$-loops $f_{\mathbb{S}_k^1}^j$ are null-homotopic in $\widetilde{\bigvee}_{i \ge m} X_i$ for $1 \le j \le l$.
Now by using \cite[Lemma 2.2]{BabMas1}, one can obtain a continuous homotopy to make $f^j$ null-homotopic. It implies that if a finite word is trivial, then all of its generating elements are trivial, or equivalently the group $\mathcal{H}_1( \mathbb{A}, \theta)$ is locally free.
%Then $G_k = \langle\{[f_k^1, f_k^2], \ldots, [f_k^m]\}\rangle$ is a subgroup of $\pi_1(\mathbb{A}, \theta)$ for each $k \in \mathbb{N}$.  Since $\pi_1(\mathbb{A})$ is locally free, all $G_k$'s are free. Then all $[f^i_k]$'s are trivial for $1\le i \le m$ and $k \in \mathbb{N}$. Thus $[f^i] \in \ker \varphi$ for $1 \le i \le m$.
\end{proof}

%There are some trivial cases which are omitted. For example archipelago spaces on simply connected spaces.
In the following theorem, assume that each $X_i$ satisfies the assumption $2 \le card \big(\pi_1(X_i, x_i) \big) \le c$. Then we present the structure of the first Hawaiian homology group of some archipelago spaces by using the fact that this group is cotorsion.

\begin{theorem}\label{th4.2nn}
Let $\mathbb{A} = \mathbb{A}(\{X_i\})$ be the archipelago space on $\{X_i\}_{i \in \mathbb{N}}$, where each $X_i$ is locally strongly contractible. Also let $\pi_1(X_i)$ be free for all $i \in I$ except a finite number. Then
\[\mathbb{H}_1
(\mathbb{A}, x_*)
\cong
\sum_{c} \mathbb{Q} \oplus \prod_{p \in P} \widehat{\sum_{c} \mathbb{J}_p}^p
\cong
\frac{\prod_{\aleph_0} \mathbb{Z}}{ \sum_{\aleph_0} \mathbb{Z}}.
\]
\end{theorem}

\begin{proof}
Since all $1$-loops at $\theta$ are small, by Lemma \ref{le0.1}, $\mathcal{H}_1(\mathbb{A}, x_*)$ is Higman-complete. Then by \cite[Lemma 2]{HerHoj}, $Ab(\mathcal{H}_1(\mathbb{A}, x_*))$ is Higman-complete. Thus by \cite[Theorem 3]{HerHoj}, $Ab(\mathcal{H}_1(\mathbb{A}, x_*))$ is cotorsion. Moreover, by Theorem \ref{th1.2}, $\mathcal{H}_1(\mathbb{A}, x_*)$ is locally free. It was shown that the abelianization of  locally free group is torsion-free \cite[Lemma 12]{HerHoj}. Therefore, the group $Ab(\mathcal{H}_1(\mathbb{A}, x_*))$ is a torsion-free cotorsion group, and then it is algebraically compact. Also, it is torsion-free algebraically compact, and then it is isomorphic to $\sum_{m_0} \mathbb{Q} \oplus \prod_{p \in P} \widehat{\sum_{m_p} \mathbb{J}_p}^p$ for some cardinalities $m_0$ and $m_p$ ($p \in P$) \cite[pp. 105,169]{Fuc}. Hence, it remains to obtain the cardinalities $m_0$ and $m_p$ ($p \in P$).
Moreover, since $X_i$ is locally strongly contractible, it is locally path-connected, and then the archipelago space $\mathbb{A}(X_i)$ is locally path-connected. Therefore by Theorem \ref{th2.10nn},
\[
\mathbb{H}_1(\mathbb{A}, x_*) \cong Ab \big( \mathcal{H}_1(\mathbb{A}, x_*) \big)
\cong
\sum_{m_0} \mathbb{Q} \oplus \prod_{p \in P} \widehat{\sum_{m_p} \mathbb{J}_p}^p.
\]
First, we show that $card \big(\mathbb{H}_1 (\mathbb{A}, x_*)\big) \le c$. Thus $m_0, m_p \le c$,  where $c$ is the continuum cardinality and $p$ is a prime.
Let $X = \widetilde{\bigvee}_{i \in I} X_i$. One can generalize the proof of \cite[Theorem 2.8]{MasMir} for the homomorphism $i_*: \mathcal{H}_1 (X, x_*)) \to \mathcal{H}_1 (\mathbb{A}, x_*)$, induced by the inclusion, to be an epimorphism. Therefore, we have $card \big(\mathcal{H}_1(X, x_*)\big) \ge card \big(\mathcal{H}_1(\mathbb{A}, x_*)\big)$. By \cite[Theorem 2.9]{BabMas1}, $card \big(\mathcal{H}_1(X, x_*)\big) \le \prod_{i\in \mathbb{N}} card \big(\pi_1(X_i, x_i)\big) \le c^{\aleph_0}= c$. Thus $card \big( \mathcal{H}_1(\mathbb{A}, x_*)\big) \le c$. Moreover, for an arbitrary pointed space $(X, x_0)$, the group $\pi_1(X, x_0)$ can be considered as a subgroup of $\mathcal{H}_1(X, x_0)$. It implies that $card\big(\mathcal{H}_1(X, x_0)\big) \ge card \big( \pi_1(X, x_0) \big)$. We know that $card\big(\pi_1(\mathbb{A}, x_*) \big) = c$ and then $card \big( \mathcal{H}_1(\mathbb{A}, x_*) \big) \ge c$.
Therefore $card \big( \mathcal{H}_1(\mathbb{A}, x_*) \big) = c$, and then we have $card \big(\mathbb{H}_1(\mathbb{A}, x_*) \big) \le c$, by Theorem \ref{th2.10nn}.
Thus $m_0, m_p \le c$.
Now by an epimorphism we show that $m_0, m_p \ge c$.

Since all $1$-loops at $\theta$ are small, by \cite[Corollary 4.3]{BabMas3}, there is an epimorphism $\varphi: \mathcal{H}_1(\mathbb{A}, \theta) \to \prod_{\aleph_0} \pi_1(\mathbb{A}, \theta)$.  Also there is an  isomorphism $\pi_1(\mathbb{A}, \theta) \to \frac{\times^{\sigma}_{\aleph_0} G}{*_{\aleph_0} G}$, where $G$ is either $\mathbb{Z}$ or $\mathbb{Z}_2$ by \cite[Theorem A]{ConHoj}.
Thus there is  an epimorphism $\mathcal{H}_1(\mathbb{HA}, \theta) \to \prod_{\aleph_0} \frac{\times^{\sigma}_{\aleph_0} G}{*_{\aleph_0} G}$.
This induces the epimorphism  $Ab \big(\mathcal{H}_1(\mathbb{A}, \theta) \big) \to Ab \big(\prod_{\aleph_0} \frac{\times^{\sigma}_{\aleph_0} G}{*_{\aleph_0} G} \big)$.
Moreover, for any family of groups $\{G_i\}_{i \in I}$, there is an epimorphism $Ab \big( \prod_{i \in I} G_i \big) \to \prod_{i \in I} Ab(G_i)$.
Since $ \mathbb{H}_1(\mathbb{A}, x_*) \cong Ab \big(\mathcal{H}_1(\mathbb{A}, \theta) \big)$, there is an epimorphism $\mathbb{H}_1(\mathbb{A}, x_*) \to \prod_{\aleph_0} Ab \big(  \frac{\times^{\sigma}_{\aleph_0} G}{*_{\aleph_0} G}\big)$. Also, $Ab \big( \frac{\times^{\sigma}_{\aleph_0} G}{*_{\aleph_0} G}\big)$ is isomorphic to $\frac{\prod_{\aleph_0} \mathbb{Z}}{ \sum_{\aleph_0} \mathbb{Z}}$ if $G$ is either $\mathbb{Z}$ or $\mathbb{Z}_2$ by \cite[Theorem 8]{HerHoj}. Hence, we have an epimorphism $\mathbb{H}_1(\mathbb{A}, x_*) \to \prod_{\aleph_0} \frac{\prod_{\aleph_0} \mathbb{Z}}{ \sum_{\aleph_0} \mathbb{Z}}$. By \cite[Corollary 42.2]{Fuc}, $\frac{\prod_{\aleph_0} \mathbb{Z}}{ \sum_{\aleph_0} \mathbb{Z}}$ is algebraically compact and it is isomorphic to
$\sum_{c} \mathbb{Q} \oplus \prod_{p \in P} \widehat{\sum_{c} \mathbb{J}_p}^p$ by \cite{Bal}.
Then there is an epimorphism $\mathbb{H}_1(\mathbb{A}, x_*) \to \prod_{\aleph_0} \sum_{c} \mathbb{Q} \oplus \prod_{p \in P} \widehat{\sum_{c} \mathbb{J}_p}^p$. Equivalently there is an epimorphism $\sum_{m_0} \mathbb{Q} \oplus \prod_{p \in P} \widehat{\sum_{m_p} \mathbb{J}_p}^p \to \prod_{\aleph_0} \bigg( \sum_{c} \mathbb{Q} \oplus \prod_{p \in P} \widehat{\sum_{c} \mathbb{J}_p}^p \bigg)$. Therefore, $m_0, m_p \ge c$, and then $m_0, m_p = c$ for $p \in P$.
Hence $\mathbb{H}_1(\mathbb{A}, x_*) \cong  \sum_{c} \mathbb{Q} \oplus \prod_{p \in P} \widehat{\sum_{c} \mathbb{J}_p}^p$, and by \cite[Corollary 42.2]{Fuc}, $\mathbb{H}_1(\mathbb{A}, x_*) \cong \frac{\prod_{\aleph_0} \mathbb{Z}}{ \sum_{\aleph_0} \mathbb{Z}}$.
%
%
% Also by \cite[Corollary 38.3]{Fuc}, $\prod_{\aleph_0} \frac{\prod_{\aleph_0} \mathbb{Z}}{ \sum_{\aleph_0} \mathbb{Z}}$ is algebraically compact. We show that
%
% and then $\prod_{\aleph_0} \frac{\prod_{\aleph_0} \mathbb{Z}}{ \sum_{\aleph_0} \mathbb{Z}}$
%
%
%
%Now we prove that $\prod_{\aleph_0} \frac{\prod_{\aleph_0} \mathbb{Z}}{ \sum_{\aleph_0} \mathbb{Z}} \cong \frac{\prod_{\aleph_0} \mathbb{Z}}{ \sum_{\aleph_0} \mathbb{Z}}$. By \cite[]{Fuc} $\frac{\prod_{\aleph_0} \mathbb{Z}}{ \sum_{\aleph_0} \mathbb{Z}}$ is algebraically compact
%
\end{proof}

Herfort and  Hojka \cite{HerHoj} proved that the first singular homology group of any archipelago space is isomorphic to $\frac{\prod_{\aleph_0} \mathbb{Z}}{ \sum_{\aleph_0} \mathbb{Z}}$. Now by Theorem \ref{th4.2nn}, the first singular homology and the first Hawaiian homology groups are isomorphic for archipelago spaces. It does not hold for many classes of spaces; for instance, see Example \ref{ex0.7}.

\section*{Acknowledgment}
This research was supported by a grant from Ferdowsi University of Mashhad--Graduate Studies (No. 2/57128). %If necessary

\end{document}